\documentclass[12pt]{amsart}
\usepackage{graphicx}
\usepackage{latexsym}
\usepackage{amssymb}
\usepackage{amsmath,amsthm,yhmath}
\usepackage{verbatim}
\usepackage{pb-diagram,amsfonts}
\usepackage{tikz}
\usepackage{hyperref}
\numberwithin{equation}{section}
\newtheorem{theorem}{Theorem}[section]
\newtheorem{Lemma}[theorem]{Lemma}
\newtheorem{corollary}[theorem]{Corollary}
\newtheorem{proposition}[theorem]{Proposition}
\theoremstyle{definition}
\newtheorem{definition}[theorem]{Definition}
\theoremstyle{remark}
\newtheorem{remark}[theorem]{Remark}

\newcommand{\R}{\mathbb{R}}
\newcommand{\Z}{\mathbb{Z}}
\newcommand{\uc}{\mathbb{S}}

\newcommand{\D}{\mathbb{D}}
\newcommand{\N}{\mathbb{N}}
\newcommand{\M}{\mathcal{M}}
\newcommand{\C}{\mathbb{C}}
\newcommand{\Ca}{\mathcal{C}}
\newcommand{\lam}{\mathcal{L}}
\newcommand{\hlam}{\widehat{\mathcal{L}}}

\newcommand{\qml}{\mathrm{QML}}
\newcommand{\sm}{\setminus}
\newcommand{\ch}{\mathrm{CH}}
\newcommand{\Bd}{\mathrm{Bd}}
\newcommand{\ol}{\overline}
\newcommand{\disk}{\mathbb{D}}
\newcommand{\cdisk}{\ol{\mathbb{D}}}
\newcommand{\sh}{\mathrm{SH}}
\newcommand{\si}{\sigma}
\newcommand{\la}{\lambda}

\newcommand{\ga}{\gamma}

\newcommand{\hell}{\hat{\ell}}
\newcommand{\tell}{\Tilde{\ell}}
\newcommand{\A}{\mathcal{A}}
\newcommand{\hc}{\hat{c}}

{\par\noindent\textit{Proof of (#1)}%
}

\begin{document}

\title{Lavaurs algorithm for cubic\\ symmetric polynomials}\footnote{\date{\today}}

\author[Blokh]{Alexander Blokh}
\author[Oversteegen]{Lex~G.~Oversteegen}
\author[Selinger]{Nikita Selinger}
\author[Timorin]{Vladlen Timorin}
\author[Vejandla]{Sandeep chowdary vejandla}

\address[Alexander~Blokh, Lex~Oversteegen, Nikita~Selinger and Sandeep~Vejandla]
{Department of Mathematics\\ University of Alabama at Birmingham\\
Birmingham, AL 35294-1170}

\address[Vladlen~Timorin]
{Faculty of Mathematics\\
HSE University, Russian Federation\\
6 Usacheva St., 119048 Moscow
}

\thanks{The second named author was partially supported by NSF-DMS-1807558}

\thanks{The fourth named author was partially supported by the HSE University Basic Research Program}





\begin{abstract}
To investigate the degree $d$ connectedness locus, Thur\-ston studied
\emph{$\sigma_d$-invariant laminations}, where $\sigma_d$ is the
$d$-tupling map on the unit circle, and built a topological model for the
space of quadratic polynomials $f_c(z) = z^2 +c$. In the same spirit, we
consider the space of all \emph{cubic symmetric polynomials}
$f_\lambda(z)=z^3+\la^2 z$ in three articles. In the first one we construct the
lamination $C_sCL$ together with the induced factor space $\mathbb{S}/C_sCL$ of
the unit circle $\mathbb{S}$. As will be verified in the third paper, $\mathbb{S}/C_sCL$
is a monotone model of the \emph{cubic symmetric connectedness locus}, i.e., the
space of all cubic symmetric polynomials with connected Julia sets. In the
present paper, the second in the series,  we develop an
algorithm for constructing $C_sCL$
analogous to the Lavaurs algorithm for constructing a combinatorial model
$\mathcal{M}^{comb}_2$ of the Mandelbrot set $\mathcal{M}_2$.
\end{abstract}

\maketitle

\section{Introduction}
We use standard notation ($\R, \C$ for the real/complex numbers, $\disk$ for
the unit disk centered at the origin, etc). The Riemann sphere is denoted by
$\hat{\C}$. The boundary (in $\C$) of a set $X\subset\C$ is denoted by
$\Bd(X)$.  We consider \emph{only complex}
polynomials $P$; 
for such a $P$, let $J_P$ be its Julia set and $K_P$ be its filled Julia set.
A \emph{chord} is a closed straight line segment with endpoints on the unit
circle $\uc=\Bd(\disk)$.

The connectedness locus $\M_d$ is the space of polynomials of degree $d$, up
to affine conjugacy, with connected Julia sets. A fundamental problem
is to understand the structure of $\M_d$. Major progress
has been made for $d=2$ but much less is known for $d>2$. Thurston
\cite{thu85} introduced geodesic invariant laminations to provide a combinatorial model for
$\mathcal M_2$. A \emph{lamination} $\lam$ is a compact set of chords, called
\emph{leaves}, 
that are pairwise disjoint in $\disk$ (equivalently, do not \emph{cross}).

More precisely, Thurston constructs a lamination $QML$ whose leaves tag all
invariant quadratic laminations (for $d\ge 2$, a lamination is
\emph{invariant} if it is invariant under the map $\sigma_d(z)=z^d$
restricted to $\uc$, see Definition \ref{inv-lam}). It can be shown that the
quotient space $\M_2^{Comb}=\uc/QML$ is a monotone image of $\Bd(\M_2)$
(conjecturally, this map is a homeomorphism), cf. \cite{thu85}.
No such models exist for $d>2$.

A natural next object of study is $\M_3$, and some slices of $\M_3$ have
already been considered. In \cite{bostv1} we construct the lamination $C_sCL$
(this stands for \emph{cubic symmetric comajor lamination}) together with the
induced factor space $\uc/C_sCL$ of the unit circle $\uc$. In \cite{bostv3}
we verify that $\uc/C_sCL$ is a monotone model of the \emph{cubic symmetric
connected locus}, i.e. the space $\mathcal M_{3,s}$ of \emph{symmetric cubic
polynomials} $P(z)=z^3+\lambda^2 z$ with connected Julia sets.

To understand the structure of $C_sCL$ and to be able to obtain suitable
pictures of this space in this paper we provide an algorithm for constructing
a dense set of leaves in $C_sCL$ (see Figure 1).

\begin{figure}
   \centering
        \includegraphics[height=5cm]{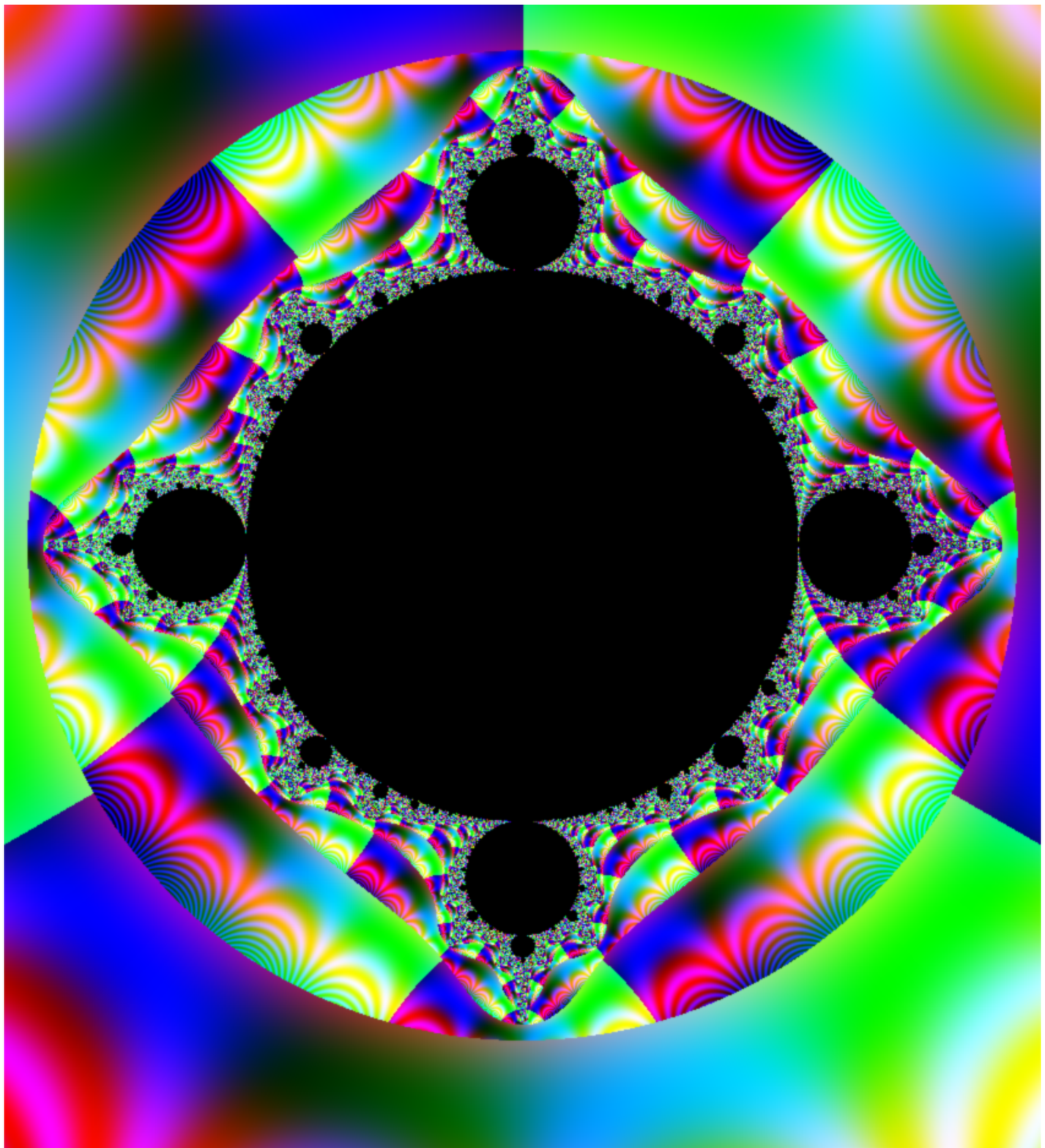}
        \includegraphics[height=5cm]{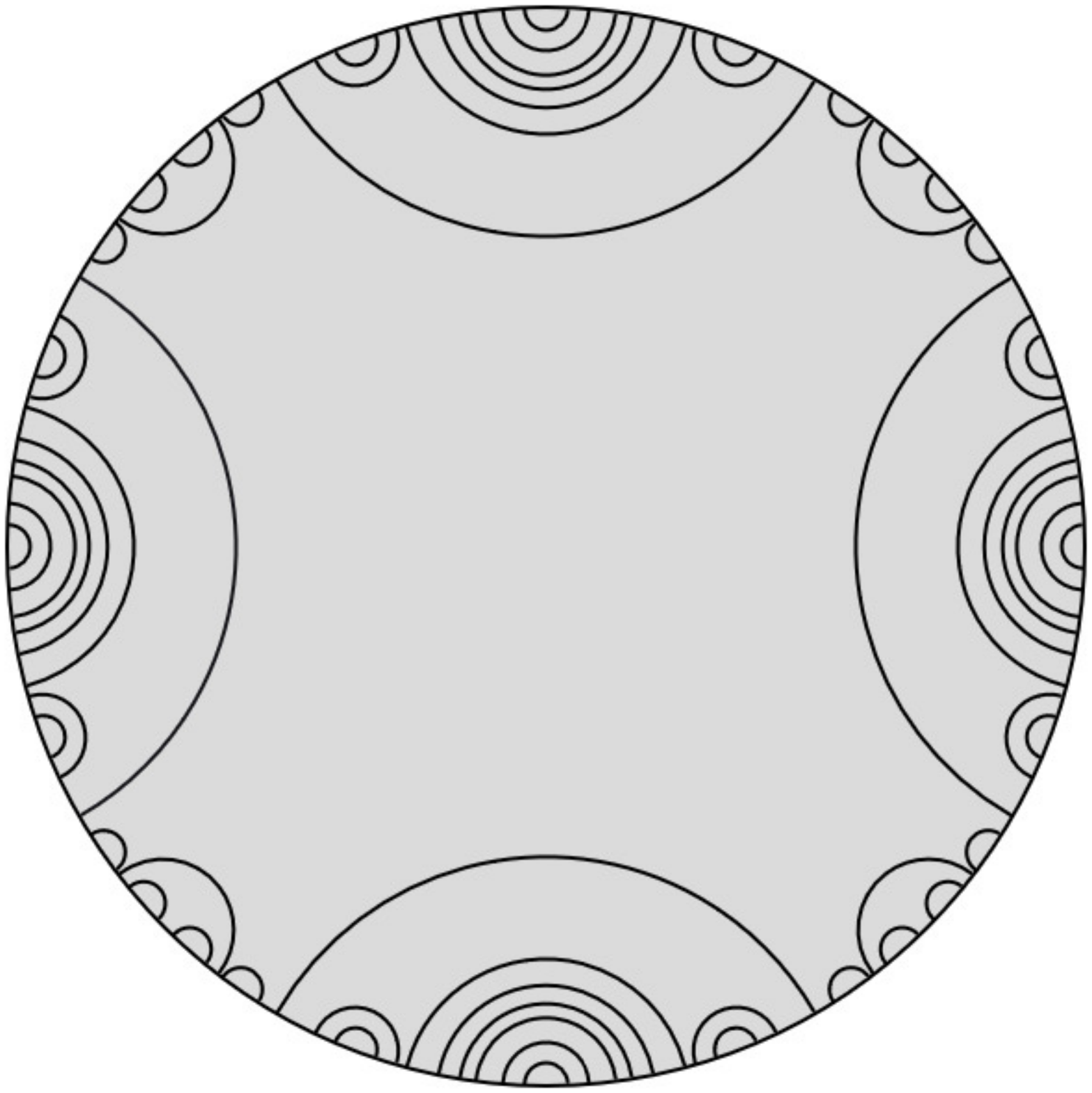}
     \caption{The parameter space of symmetric cubic  polynomials $\mathcal M_{3,s}$ on the left and the   Symmetric Cubic Comajor Lamination $C_sCL$ on the right.}
\end{figure}

\section{Laminations: classical definitions}

\subsection{Invariant laminations}

The reader is referred to \cite{mil00,thu85} for basic notions of complex polynomial dynamics on $\C$,
 including Fatou and Julia sets, external rays, landing etc.

Identify $\uc$ with $\R/\Z$ and define the map $\sigma_d:\uc \rightarrow \uc$ for $d \geq 2$
 as $\sigma_d(z)=dz \mod 1$; clearly, $\sigma_d$ is locally one-to-one on $\mathbb{S}$.
A complex polynomial $P$ with locally connected Julia set $J_P$ gives rise to
 an equivalence relation $\sim_P$ on $\uc$ so that $x \sim_P y$ if and only if
 the external rays of arguments $x$ and $y$ land at the same point of $J_P$.
Equivalence classes of $\sim_P$ have pairwise disjoint convex hulls.
The \emph{topological Julia set} $\uc/\sim_P=J_{\sim_P}$ is homeomorphic to $J_P$,
 and the \emph{topological polynomial} $f_{\sim_P}:J_{\sim_P}\to J_{\sim_P}$, induced by $\sigma_d$,
 is topologically conjugate to $P|_{J_P}$.

For a closed set $A\subset \uc$ we denote its convex hull by $\ch(A)$.
An \emph{edge} of $\ch(A)$ is a chord of $\uc$ contained in the boundary of $\ch(A)$.
Given points $a$, $b\in\uc$, let $(a,b)$ be the positively oriented arc in $\uc$ from $a$ to $b$
 and $\ol{ab}$ be the chord with endpoints $a$ and $b$.

\begin{definition}\label{lam}
A \emph{lamination} $\lam$ is a set of chords in the closed unit disk
$\overline{\D}$, called \emph{leaves} of $\lam$, if it satisfies the
following conditions:

\noindent (L1) leaves of $\lam$ do not cross; \, (L2) the set
$\lam^{\ast}=\cup_{\ell\in\lam}\ell$ is closed.\\ If (L2) is not assumed then
$\lam$ is called a \emph{prelamination}.
\end{definition}

A degenerate leaf is a point of $\uc$. Given a leaf $\ell =\overline{ab}\in
\lam$, let $\sigma_d(\ell)$ be the chord with endpoints $\sigma_d(a)$ and
$\sigma_d(b)$; then $\ell$ is called a \emph{pullback} of $\si_d(\ell)$. If
$a\ne b$ but $\sigma_d(a) = \sigma_d(b)$, call $\ell$ a {\em critical leaf}.
Let 
$\sigma_d^{\ast}:\lam^{\ast}\rightarrow\overline{\D}$ be the linear extension
of $\sigma_d$ over all the leaves in $\lam$. Then $\sigma_d^{\ast}$ is
continuous and $\sigma_d^{\ast}$ is one-to-one on any non-critical leaf. If
$\lam$ 
includes all points of $\uc$ as degenerate
leaves, then $\lam^{\ast}$ is a continuum.

\begin{definition}[Gap]\label{gap-dfn}
A {\em gap} $G$ of a lamination $\lam$ is the closure of a
component of $\D\sm\lam^{\ast}$; its boundary
leaves are called \emph{edges (of the gap)}.
\end{definition}

If $G$
is a leaf/gap of $\lam$, then $G$ equals the convex hull of $G\cap \uc$. If $G$
is a leaf or a gap of $\lam$ we let $\sigma_d(G)$ be the convex hull of
$\sigma_d(G\cap \uc)$. Notice that $\Bd(G) \cap \uc = G \cap \uc$. Points of $G\cap \uc$ are called the \emph{vertices} of $G$.
A gap $G$ is called \emph{infinite (finite)} if and only if $G\cap \uc$ is
infinite (finite). A gap $G$ is called \emph{triangular gap} if $G\cap \uc$
consists of three points.

\begin{definition}\label{d:gen-lam}
Let $\lam$ be a lamination. The equivalence relation $\sim_\lam$ induced by
$\lam$ is defined by declaring that $x\sim_\lam y$ if and only if there
exists a finite concatenation of leaves of $\lam$ joining $x$ to $y$.
\end{definition}


\begin{definition}[Invariant (pre)laminations]\label{inv-lam} A (pre)lamination $\lam$ is
\emph{($\sigma_d$-)in\-va\-ri\-ant} if,

\noindent (D1) $\lam$ is \emph{forward invariant.} For each $\ell \in \lam$
either $\sigma_d(\ell) \in \lam$ or $\sigma_d(\ell)$ is a point in $\uc$ and

\noindent (D2) $\lam$ is \emph{backward invariant.}
\begin{enumerate}
    \item For each $\ell \in \lam$ there exists a leaf $\ell' \in \lam$
        such that $\sigma_d(\ell')=\ell$.
    \item  For each $\ell \in \lam$ such that $\sigma_d(\ell)$ is a
        non-degenerate leaf, there exists \textbf{d disjoint} leaves
        $\ell_1,......\ell_d$ in $\lam$ such that $\ell = \ell_1$ and
        $\sigma_d(\ell_i)=\sigma_d(\ell)$ for all i.
\end{enumerate}{}
\end{definition}{}

\begin{definition}[q-lamination] A $\si_d$-invariant lamination $\lam$ is called a
\emph{q-lamination} if $x\sim_\lam y$ implies that $x$ and $y$ are vertices of the same finite gap or leaf.
The convex hulls of $\sim_\lam$-classes are also called \emph{$\sim_\lam$-sets} or \emph{$\lam$-sets}.
\end{definition}{}

\begin{remark}
It readily follows from the definition of a q-lamination that
at most two leaves of a q-lamination can share an endpoint.
\end{remark}

\begin{definition}[Siblings]\label{siblings} Two chords are called
\emph{siblings}  if they have the same image. Any $d$ disjoint chords with
the same non-degenerate image are called a \emph{sibling collection}.
\end{definition}{}

\begin{definition}[Monotone Map]
Let $X$, $Y$ be topological spaces and $f:X\rightarrow Y$ be
continuous. Then $f$ is said to be {\em monotone} if $f^{-1}(y)$ is
connected for each $y \in Y$. It is known that if $f$ is monotone and
$X$ is a continuum then $f^{-1}(Z)$ is connected for every connected
$Z\subset f(X)$.
\end{definition}

\begin{definition}[Gap-invariance] A lamination $\lam$ is \emph{gap invariant}
if for each gap $G$, the set $\sigma_d(G)$ is a gap, or a leaf, or a single point.
In the first case we also require that $\sigma_d^*|_{\Bd(G)}:\Bd(G)\to
\Bd(\sigma_d(G))$  maps as the composition of a monotone map and a covering
map to the boundary of the image gap, with positive orientation (i.e., as you
move through the vertices of $G$ in clockwise direction around $\Bd(G)$,
their corresponding images in $\sigma_d(G)$ must also be aligned clockwise in
$\Bd(\sigma_d(G))$).
\end{definition}{}

\begin{definition}[Degree]
 The \emph{degree} of the map $\sigma_d^*|_{\Bd(G)}:\Bd(G)\to \Bd(\sigma_d(G))$,
 or of the gap $G$, is defined as the number of
 components of $(\sigma_d^*)^{-1}(x)$ in $\Bd(G)$, for any $x\in \Bd(\sigma_d(G))$.
In other words,
if every leaf of $\sigma_d(G)$ has $k$ disjoint pre-image leaves in $G$, then
the degree of the map $\sigma_d^*$ is $k$. A gap $G$ is called
\emph{critical} gap if $k>1$.
\end{definition}

The following results are proved in \cite{bmov13}.

\begin{theorem}
Every \emph{($\sigma_d$-)invariant} lamination is gap invariant.
\end{theorem}{}

\begin{theorem}\label{t:prelam-cl}
The closure of an invariant prelamination is an invariant lamination. The
space of all $\sigma_d$-invariant laminations is compact.
\end{theorem}{}

\section{Parameter lamination $C_sCL$: preliminaries}

This section describes results of \cite{bostv1}. From now on normalize the
circle so that its length is $1$; the length of arcs and angles are measured
accordingly. Given a chord $\ell=\ol{ab}$ denote by $-\ell$ the chord
obtained by rotating $\ell$ by the angle $\frac12$. Define the \emph{length}
$\|\ol{ab}\|$ of a chord $\ol{ab}$ as the shorter of the lengths of the arcs
in $\uc=\R/\Z$ with the endpoints $a$ and $b$. The maximum length of a chord
is $\frac12$. Divide leaves into three categories by their length.

\begin{definition} A \textit{short leaf} is a leaf $\ell$ such that $0<\|\ell\|<\frac{1}{6}$,\\
a \textit{medium leaf} is a leaf $\ell$ such that $\frac{1}{6}\leq\|\ell\|<\frac{1}{3}$ and \\
a \textit{long leaf} is a leaf $\ell$ such that
$\frac{1}{3}<\|\ell\|\leq\frac{1}{2}$.
\end{definition}

Critical leaves are exactly leaves of length $\frac13$. A leaf $\ell$ is
long/medium if $\|\ell\|\ge \frac16$.
A non-critical leaf $\ell$ of $\lam$ has \emph{siblings} (Definition
\ref{siblings}). A non-critical long/medium leaf $\ell$ has a long/medium
sibling $\hat{\ell}$, and there is a unique component $C^\circ(\ell)$ of
$\D\setminus (\ell\cup\hat{\ell})$ whose boundary contains $\ell\cup
\hat{\ell}$.

\begin{Lemma}[Lemma 3.4 \cite{bostv1}]\label{l:sml}
The possibilities for leaves in a sibling collection are
\begin{description}
    \item[(sss)] all leaves are \textit{short};
    \item[(mmm)] all leaves are \textit{medium};
    \item[(sml)] one leaf is \textit{short}, one \textit{medium}, and one
        \textit{long}.
\end{description}
A sibling collection is completely determined by its type and one leaf.
\end{Lemma}

These are general facts; let us now become more specific.

\begin{definition}[Cubic symmetric lamination] A $\sigma_3$-invariant
lamination $\lam$ is called a \emph{cubic symmetric lamination} if:

\noindent (D3) for each $\ell \in \lam$ we have $-\ell\in \lam$.
\end{definition}{}

Unless otherwise stated, let \emph{$\lam$ be a cubic symmetric lamination}.

\begin{definition}\label{d:sibli}
Suppose that $\ell=\ol{ab}$ is a non-critical chord which is not a diameter
and the arc $(a, b)$ is shorter than the arc $(b, a)$. Denote the chord
$\ol{(a+\frac13) (b-\frac13)}$ by $\ell'$ and the chord $\ol{(a+\frac23)
(b-\frac23)}$ by $\ell''$.
\end{definition}

As $\si_3(\ell')=\si_3(\ell'')=\si_3(\ell)$, $\{\ell, \ell', \ell''\}$ is a
sibling collection. For a long/medium non-critical leaf $\ell\in \lam$, it
follows that $\ell'$ is long/medium and $\ell''$ is short; if, moreover,
$\ell\in \lam$ (recall that $\lam$ is a symmetric lamination), its sibling
collection is $\{\ell, \ell', \ell''\}$ (all other possibilities lead to
crossings with $\ell$ or $-\ell$). So, for $\lam$ a sibling collection of
type (mmm) is impossible.

\begin{definition}\label{d:strip} Given two chords $\ell, \hell$ that do not
cross let $\mathcal{S}(\ell, \hell)$ be a component of $\disk\sm [\ell\cup
\hell]$ with boundary containing $\ell$ and $\hell$; call $\mathcal{S}(\ell,
\hell)$ the \emph{strip between $\ell$ and $\hell$}.
\end{definition}

The above notation is convenient when dealing with laminations.

\begin{definition} [Short strips]
For a sibling collection $\{\ell,\ell', \ell''\}$ of type (sml), with $\ell$
and $\ell'$ long/medium, let $C(\ell)=\ol{\mathcal{S}(\ell, \ell')},$ (the
short leaf $\ell''$ cannot lie in $C(\ell)$). The set $C(\ell)$ has two
boundary circle arcs of length $| \frac13-\|\ell\| |$ (and so does
$-C(\ell)$). Given a long/medium chord $\ell\in \lam$, call the region
$\sh(\ell)=C(\ell)\cup -C(\ell)$ the \emph{short strips (of $\ell$)} and each
of $C(\ell)$ and $-C(\ell)$ a \emph{short strip (of $\ell$).} Call $|
\frac13-\|\ell\| |=w(C(\ell))=w(-C(\ell))=w(\sh(\ell))$ the \emph{width} of
$C(\ell)$ (or of $-C(\ell)$, or of $\sh(\ell)$). Note that
$-C(\ell)=C(-\ell)$.
\end{definition}

\begin{definition}
A leaf $\ell$ is \emph{closer to criticality} than a leaf $\hell$ if
$\|\ell\|$ is closer to $\frac13$ than $\|\hell\|$. A chord $\ell$ is
\emph{closest to criticality (in a family of chords $\mathcal A$)} if its
length is the closest to criticality among lengths of chords from $\mathcal
A$. 
\end{definition}

The next two facts established in \cite{bostv1} are similar to important
results proven in \cite{thu85}. The first one is somewhat technical.

\begin{proposition}[Lemma 3.7 \cite{bostv1}]\label{str-prop}
If $\ell\in \lam,$ $\|\ell\|>\frac16$, and $k\in \N$ is minimal such that
$\ell_k=\sigma_3^k(\ell)$ intersects the interior of $\sh(\ell)$, then
$\|\ell_k\|>\frac16$, and $\ell_k$ is closer to criticality than $\ell$.
A leaf $\ell$ that is the closest to criticality in its forward orbit
 is medium/long, and no forward image of $\ell$ enters the interior of $\sh(\ell)$.
\end{proposition}

Proposition \ref{str-prop} implies Theorem \ref{nwt-thm}.

\begin{theorem}[Theorem 3.8 \cite{bostv1}]\label{nwt-thm} Let $\lam$ be
a symmetric lamination and $G$ be a gap of $\lam$. Then $G$ is preperiodic
unless an eventual forward image of $G$ is a leaf or a point.
\end{theorem}

Call a finite periodic gap of $\lam$ a \emph{periodic polygon}.

\begin{Lemma}[Lemma 4.5 \cite{bostv1}]\label{fn-gaps}
Let $G$ be a periodic polygon, and let $g$ be the first return map of $G$.
One of the following is true.

$(a)$ The first return map $g$ acts on the sides of $G$ transitively
    as a rational rotation.

$(b)$ The edges of $G$ form two disjoint periodic orbits,
and $G$ eventually maps to the
    gap $-G$. If $\ell$ and $\hell$ are two adjacent edges of $G$, then the
    leaf $\ell$ eventually maps to the edge $-\hell$ of $-G$.
\end{Lemma}

\begin{definition}
If case (a) from Lemma \ref{fn-gaps} holds, we call a gap $G$ a
\emph{1-rotational gap}. If case (b) from Lemma \ref{fn-gaps} holds we call
such a gap \emph{a 2-rotational gap}.
\end{definition}

If $c$ is a short chord, then there are two long/medium chords with the same
image as $c$. We will denote them by $M_c$ and $M'_c$. Also, denote by $Q_c$
the convex hull of $M_c\cup M'_c$. This applies in the degenerate case, too:
if $c\in \uc$ is just a point, then $M_c=M'_c=Q_c$ is a critical leaf $\ell$
disjoint from $c$ such that $\si_3(c)=\si_3(M_c)$.

\begin{definition}
[Major] A leaf $M\in \lam$ closest to criticality in $\lam$ is called a
\emph{major} of $\lam$.
\end{definition}

If $M$ is a major of $\lam$, then the medium/long sibling $M'$ of $M$ is also
a major of $\lam$, as well as the leaves $-M$ and $-M'$. A lamination has
either exactly 4 non-critical majors or 2 critical majors.

\begin{definition}[Comajor] The short siblings of  major leaves of $\lam$ are
called  \emph{comajors}; we also say that they form a \emph{comajor pair}. A
pair of symmetric chords is called a \emph{symmetric pair}. If the chords are
degenerate, their symmetric pair is called \emph{degenerate}, too.
\end{definition}

A symmetric lamination has a symmetric pair of comajors $\{c,-c\}$.

\begin{definition}
[Minor] Images of majors (equivalently, comajors) are called \emph{minors} of
a symmetric lamination. Similarly to comajors, every symmetric lamination has
two symmetric minors $\{m,-m\}$.
\end{definition}

Critical majors of a lamination have no siblings, and we define
\emph{degenerate} comajors and minors as corresponding points on $\uc$. If
majors $M$ and $-M$ are non-critical, then there is a critical gap, say, $G$
with edges $M$ and $M'$, and a critical gap $-G$ with edges $-M$ and $-M'$.


\begin{Lemma}[Lemma 5.4 \cite{bostv1}]\label{short-leaves}
Let $\{m,-m\}$ be the minors of $\lam$, and let $\ell$ be a leaf of $\lam$.
Then no
forward image of $\ell$ is shorter than $\min(\|\ell\|, \|m\|)$.
\end{Lemma}

\begin{definition}\label{d:2sided}
For a family $\mathcal{A}$ of chords, $\ell$ is a \emph{two sided limit leaf}
of $\mathcal{A}$ if $\ell$ is approximated by chords of $\mathcal{A}$ from
both sides.
\end{definition}

\begin{Lemma}[Lemma 5.5 \cite{bostv1}]\label{cmajor-end-points}
Let $c$ be a comajor and $M$ be a major of $\lam$ such that
$\si_3(c)=\si_3(M)$.

\begin{enumerate}

\item If $c$ is non-degenerate, then one of the following holds:

\begin{enumerate}

\item the endpoints of $c$ are both strictly preperiodic with the same
    preperiod and period;

\item the endpoints of $c$ are both not preperiodic, and $c$ is
    approximated from both sides by leaves of $\lam$ that have no common
    endpoints with $c$.

\end{enumerate}

\item If $M$ is non-critical, then its endpoints are either both periodic
    or both strictly preperiodic with the same preperiod and period, or
    both not preperiodic.

\end{enumerate}

In particular, a non-degenerate comajor is not periodic.

\end{Lemma}

Comajors can be described more explicitly.

\begin{definition}[Legal pairs, Definition 5.6 \cite{bostv1}]\label{d:legal} Let a symmetric pair
$\{c,-c\}$ be either degenerate or satisfy the following:

\begin{enumerate}
    \item[(a)] no two iterated forward images of $c$ and $-c$ cross, and 

    \item[(b)] no forward image of $c$ crosses the interior of $\sh(c)$.
\end{enumerate}

Then $\{c, -c\}$ is said to be a \emph{legal pair}.

\end{definition}

We will also need an important concept of a \emph{pullback} of a set.

\begin{definition}[Pullbacks, Definition 5.7 \cite{bostv1}]\label{d:pullback}
Suppose that a family $\A$ of chords is given and $\ell$ is a chord. A
\emph{pullback chord of $\ell$ generated by $\A$} is a chord $\ell'$ with
$\si_3(\ell')=\ell$ such that $\ell'$ that does not cross chords from $\A$.
An \emph{iterated pullback chord of $\ell$ generated by $\A$} is a pullback
chord of an (iterated) pullback chord of $\ell$.
\end{definition}

Lemma \ref{l:16} considers two specific cases.

\begin{Lemma}[Lemma 5.8 \cite{bostv1}]\label{l:16}
The only two symmetric laminations with comajors of length $\frac16$ have two
critical Fatou gaps 
and are as follows.

\noindent $(1)$ $\lam_1$ has the comajor pair $(\ol{\frac16 \frac13}$,
$\ol{\frac23 \frac56})$. The gap $U_1'$ is invariant; $U'_1\cap \uc$ consists
of all $\ga\in \uc$ such that $\si_3^n(\ga)\in [0, \frac12]$. The gap $U''_1$
is invariant; $U'_1\cap \uc$ consists of all $\ga\in \uc$ such that
$\si_3^n(\ga)\in [\frac12, 0]$. The gaps $U_1', U_1''$ share an edge $\ol{0
\frac12}$; their edges are the appropriate pullbacks of $\ol{0\frac12}$ that
never separate $\ol{\frac16 \frac13}$, $\ol{\frac23 \frac56}$, and
$\ol{0\frac12}$.

\noindent $(2)$ $\lam_2$ has the comajor pair
$(\ol{\frac{11}{12}\frac{1}{12}}$, $\ol{\frac{5}{12} \frac{7}{12}})$. The
gaps $U'_2, U''_2$ form an orbit and the set $(U_2'\cup U_2'')\cap \uc$
consists of all $\ga\in \uc$ such that $\si_3^n(\ga)\in [\frac1{12},
\frac5{12}]\cup [\frac7{12}, \frac{11}{12}]$. The gaps $U_2', U_2''$ share an
edge $\ol{\frac14 \frac34}$; their edges are the appropriate pullbacks of
$\ol{\frac14 \frac34}$ that never separate $\ol{\frac{11}{12}\frac{1}{12}}$,
$\ol{\frac{5}{12} \frac{7}{12}}$ and $\ol{\frac14 \frac34}$.
\end{Lemma}

Though the laminations from Lemma \ref{l:16} are not the pullback laminations
described below, knowing them allows us to consider only legal pairs with
comajors of length less than $\frac16$ and streamline the proofs.

\begin{center} \textbf{Construction of a symmetric pullback lamination $\lam(c)$ for a legal pair $\{c, -c\}$.} \end{center}

\noindent\textbf{Degenerate case.} For $c\in\uc$, let $\pm \ell=\pm M_c$.
 (call $\ell$, $-\ell$ and their pullbacks
``leaves'' even though we apply this term to existing laminations, and we are
only constructing one). Consider two cases.

(a) If $\ell$ and $-\ell$ do not have periodic endpoints, then the family of
all iterated pullbacks of $\ell, -\ell$ generated by $\ell, -\ell$ is denoted
by $\mathcal{C}_c$.

(b) Suppose that $\ell$ and $-\ell$ have periodic endpoints $p$ and $-p$.
Then there are two similar cases. First, the orbits of $p$ and $-p$ may be
distinct (and hence disjoint). Then iterated pullbacks of $\ell$ generated by
$\ell$, $-\ell$ are well-defined (unique) until 
 the $n$-th step ($n$ equals
the period of $p$ \emph{and} the period of $-p$), when there are two iterated
pullbacks of $\ell$ that have a common endpoint $x$ and 
 share other endpoints with $\ell$. 
Two other iterated pullbacks of $\ell$ 
 have a common endpoint $y\ne 0$ and 
 share other endpoints with $\ell$.
These four iterated pullbacks of $\ell$ form a collapsing quadrilateral $Q$
with diagonal $\ell$;
 moreover, $\si_3(x)=\si_3(y)$ and $\si_3^n(x)=\si_3^n(y)=z$ is
 the non-periodic endpoint of $\ell$. Evidently, $\si_3(Q)=\ol{\si_3(p)\si_3(x)}$ is
the $(n-1)$-st iterated pullback of $\ell$. Then in the pullback lamination
$\lam(c)$ that we are defining we postulate the choice of \emph{only the
short pullbacks} among the above listed iterated pullbacks of $\ell$. So,
only two short edges of $Q$ are included in the set of pullbacks
$\mathcal{C}_c$. A similar situation holds for $-\ell$ and its iterated
pullbacks.

In general, the choice of pullbacks of the already constructed leaf $\hell$
is ambiguous only if $\hell$ has an endpoint $\si_3(\pm \ell)$. In this case
we \emph{always} choose short pullbacks of $\hell$. Evidently, this defines a
set $\mathcal{C}_c$ of chords in a unique way.

We claim that $\mathcal{C}_c$ is an invariant prelamination. To show that
$\mathcal{C}_c$ is a prelamination we need to show that its leaves do not
cross. Suppose otherwise and choose the minimal $n$ such that $\hell$ and
$\tell$ are pullbacks of $\ell$ or $-\ell$ under at most the $n$-th iterate
of $\si_3$ that cross. By construction, $\hell, \tell$ are not critical.
Hence their images $\si_3(\hell),$ $\si_3(\tell)$ are not degenerate and do
not cross. It is only possible if $\hell, \tell$ come out of the endpoints of
a critical leaf of $\lam$. We may assume that $\|\hell\|\ge \frac16$ (if
$\hell$ and $\tell$ are shorter than $\frac16$ then they cannot cross).
However by construction this is impossible. Hence $\mathcal{C}_c$ is a
prelamination. The claim that $\mathcal{C}_c$ is invariant is
straightforward; its verification is left to the reader. By Theorem
\ref{t:prelam-cl}, the closure of $\mathcal{C}_c$ is an invariant lamination
denoted $\lam(c)$. Moreover, by construction $\mathcal{C}_c$ is symmetric
(this can be easily proven using induction on the number of steps in the
process of pulling back $\ell$ and $-\ell$). Hence $\lam(c)$ is a symmetric
invariant lamination.

\noindent\textbf{Non-degenerate case.} As in the degenerate case, we will
talk about leaves even though we are still constructing a lamination. By
Lemma \ref{l:16}, we may assume that $|c|<\frac16$. Set $\pm M=\pm M_c, \pm
Q=\pm Q_c$. If $d$ is an iterated forward image of $c$ or $-c$, then, by
Definition \ref{d:legal}(b), it cannot intersect the interior $Q$ or $-Q$.
Consider the set of leaves $\mathcal{D}$ formed by the edges of $\pm Q$ and
$\bigcup_{m=0}^{\infty}\{\sigma_3^{m}(c),\sigma_3^{m}(-c)\}$. It follows
that leaves of $\mathcal{D}$ do not cross among themselves. The idea is to
construct pullbacks of leaves of $\mathcal{D}$ in a step-by-step fashion and
show that this results in an invariant prelamination $\mathcal{C}_c$ as in
the degenerate case.

More precisely, we proceed by induction. Set $\mathcal{D}=\mathcal{C}_c^0$.
Construct sets of leaves $\mathcal{C}_c^{n+1}$ by collecting pullbacks of
leaves of $\mathcal{C}_c^{n}$ generated by $Q$ and $-Q$ (the step of
induction is based upon Definition \ref{d:legal} and Definition
\ref{d:pullback}). The claim is that except for the property (D2)(1) from
Definition \ref{inv-lam} (a part of what it means for a lamination to be
backward invariant), the set $\mathcal{C}_c^n$ has all the properties of
invariant laminations listed in Definition \ref{inv-lam}. Let us verify this
property for $\mathcal{C}_c^1$. Let $\ell\in \mathcal{C}_c^1$. Then
$\si_3(\ell)\in \mathcal{D}$, so property (D1) from Definition \ref{inv-lam}
is satisfied. Property (D2)(2) is, evidently, satisfied for edges of $Q$ and
$-Q$. If $\ell$ is not an edge of $\pm Q$, then, since leaves $\pm
\si_3(Q)=\si_3(\pm c)$ do not cross $\si(\ell)$, and since on the closure of
each component of $\uc\sm [Q\cup -Q]$ the map is one-to-one, then $\ell$ will
have two sibling leaves in $\mathcal{C}_c^1$ as desired. Literally the same
argument works for $\ell\in \mathcal{C}_c^{n+1}$ and proves that each set
$\mathcal{C}_c^{n+1}$ has properties (D1) and (D2)(2) from Definition
\ref{inv-lam}. This implies that $\bigcup_{i\ge 0}
\mathcal{C}_c^i=\mathcal{C}_c$ has all properties from Definition
\ref{inv-lam} and is, therefore, an invariant prelamination. By Theorem
\ref{t:prelam-cl}, its closure $\lam(c)$ is an invariant lamination.

The lamination $\lam(c)$ is called the \emph{pullback lamination (of $c$);}
we often use $c$ as the argument, instead of the less discriminatory $\{c,
-c\}$.

\begin{Lemma}[Lemma 5.9 \cite{bostv1}]\label{pull-back-lam1}
A legal pair $\{c,-c\}$ is the comajor pair of the  lamination $\lam(c)$. A
symmetric pair $\{c, -c\}$ is a comajor pair if and only if it is legal.
\end{Lemma}

Theorem \ref{t:oldmain} summarizes the main results of \cite{bostv1}.
\emph{Co-periodic comajors} are defined as preperiodic comajors of preperiod 1.
The name is due to the fact that a co-periodic comajor is a sibling of a periodic major.

\begin{theorem}\label{t:oldmain}
The set of all comajors of cubic symmetric laminations is a q-lamination.
Co-periodic comajors are disjoint from all other comajors.
\end{theorem}

Based upon this theorem we define the main object of our interest.

\begin{definition}
\label{d:cscl} All comajors of cubic symmetric laminations form a lamination
$C_sCL$ called the \emph{Cubic symmetric Comajor Lamination}.
\end{definition}

The following useful notation is justified by Theorem \ref{t:oldmain}.

\begin{definition}\label{under-dfn}
For a non-diameter chord $n=\overline{ab}$, the smaller of the two arcs into
which $n$ divides $\uc$, is denoted by $H(n)$. Denote the closed subset of
$\overline{\D}$ bounded by $n$ and $H(n)$ by $R(n)$. Given two comajors $m$
and $n$, write $m\prec n$ if $m\subset R(n)$, and say that $m$ is
\emph{under} $n$.
\end{definition}

\begin{Lemma}[Lemma 5.14 \cite{bostv1}]\label{cmajor-comp-d}
Let $\{c,-c\}$ and $\{d,-d\}$ be legal pairs, where $c$ is degenerate and
$c\prec d$. Suppose that $c$ is not an endpoint of $d$, or $\si_3(c)$ is not
periodic. Then $d\in \lam(c)$. In addition, the following holds.

\begin{enumerate}

\item Majors $D$, $D'$ of $\lam(d)$ are leaves of $\lam(c)$ unless $\lam(c)$
    has two finite gaps $G, G'$ that contain $D, D'$ as their diagonals,
    share a critical leaf $M$ of $\lam(c)$ as a common edge, and are such
    that $\si_3(G)=\si_3(G')$ is a preperiodic gap.

\item If majors of $\lam(d)$ are leaves of $\lam(c)$ and $\ell\in\lam(d)$
    is a leaf that never maps to a short side of a collapsing quadrilateral
    of $\lam(d)$, then $\ell\in\lam(c)$.

\end{enumerate}

\end{Lemma}

\section{Fatou conjecture on density of hyperbolicity}\label{s:fatou}

Co-periodic comajors correspond to periodic majors.
In Section \ref{s:fatou} we associate them with q-laminations with periodic Fatou gaps
of degree greater than 1 and show that these are dense.

\begin{definition}\label{d:hyperb}
If a  symmetric lamination $\lam$  has a periodic Fatou gap of degree greater
than $1$ (i.e., if it has properties listed in Lemma \ref{l:hyper1}), then
$\lam$ is called \emph{hyperbolic}.
\end{definition}

We need a result of \cite{bmov13}. Recall that, as in Definition
\ref{d:gen-lam}, a lamination $\lam$ generates an equivalence relation
$\sim_\lam$ on $\uc$ by declaring that $a\sim_\lam b$ if and only if a finite
concatenation of leaves of $\lam$ connects points $a\in \uc$ and $b\in \uc$.

\begin{definition}[Proper lamination, Definition 4.1 \cite{bmov13}]
Two leaves with a common endpoint $v$ and the same image which is a leaf (and
not a point) are said to form a \emph{critical wedge} (the point $v$ then is
said to be its vertex). A lamination $\lam$ is \emph{proper} if it contains
no critical leaf with periodic endpoint and no critical wedge with periodic
vertex.
\end{definition}

Proper laminations generate laminational equivalence relations.

\begin{theorem}[Theorem 4.9 \cite{bmov13}]\label{t:nowander}
Let $\lam$ be a proper invariant lamination. Then $\sim_\lam$ is an invariant
laminational equivalence relation.
\end{theorem}

We also need a nice result due to Jan Kiwi \cite{kiw02}.

\begin{theorem}[\cite{kiw02}]
\label{t:kiwi} Let $\lam$ be a $\si_d$-invariant lamination. Then any
infinite gap of $\lam$ is (pre)periodic. For any finite periodic gap $G$ of
$\lam$ its vertices belong to at most $d-1$ distinct cycles except when $G$
is a fixed return $d$-gon. In particular, a cubic lamination cannot have a
fixed return $n$-gon for $n>3$. Moreover, if all images of a $k$-gon $G$ with
$k>d$ have at least $d+1$ vertices then $G$ is preperiodic.
\end{theorem}

Finally, here is an important claim.

\begin{corollary}[Corollary 4.8 \cite{bostv1}] \label{no-diagonal}
If $E$ is a preperiodic polygon of a symmetric lamination such that $E$ is
not precritical, then no diagonal of $E$ can be a leaf of a symmetric
lamination.
\end{corollary}

Let us now describe laminations related to co-periodic comajors.

\begin{Lemma}\label{l:hyper1} Let $\lam$ be a symmetric lamination with a
periodic Fatou gap of degree greater than $1$. Then $\lam$ has two critical
Fatou gaps of degree greater than $1$. Moreover, $\lam$ is a q-lamination.
\end{Lemma}

\begin{proof}
Because of the symmetry, a hyperbolic symmetric lamination $\lam$ has two
critical Fatou gaps of degree greater than $1$. These gaps either belong to
the same cycle of Fatou gaps, or belong to two distinct cycles of Fatou gaps.
Moreover, by Theorem \ref{t:nowander} the equivalence $\sim_\lam$ is
laminational. We claim that $\lam$ coincides with the q-lamination $\hlam$
generated by $\sim_\lam$. We need to show that any leaf of $\lam$ is a leaf
of $\hlam$.

In general, edges of a Fatou gap $U$ may form a finite concatenation in which
case $U$ is not a gap of the corresponding q-lamination (by definition, in
the q-lamination we add one more leaf to the concatenation to make it into a
finite gap; this extra leaf will be an edge of a new, smaller Fatou gap
of the q-lamination). However this cannot happen in our case: if it did it
would yield a symmetric q-lamination with fixed return finite gaps
contradicting Lemma \ref{fn-gaps}. Hence the Fatou gaps of $\lam$ are gaps of
$\hlam$. Otherwise, if $\ell\in \lam$ is not a leaf of $\hlam$ then $\ell$
must be a diagonal of a finite gap $G$ of $\hlam$. However by Corollary
\ref{no-diagonal} this is impossible. Hence $\lam=\hlam$ is a q-lamination as
desired.
\end{proof}

Hyperbolic laminations are constructed in Theorem \ref{non-degen-cmajors}.

\begin{theorem}\label{non-degen-cmajors}
A preperiodic point $q\in \uc$ of preperiod $1$ and period $k$ is an endpoint
of a non-degenerate co-periodic comajor $c$ of period $k$ of a cubic symmetric lamination.
Take the short edges of $\pm Q_c$, and remove their
backward orbits from $\lam(c)$. Then the resulting lamination $\hlam(c)$ is a
hyperbolic q-lamination with comajor pair $\{c, -c\}$.
\end{theorem}

\begin{proof}
Let $\ell=M_q=\ol{x_0p}$ be the critical leaf with $\si_3(\ell)=\si_3(q)$ and
$k$-periodic endpoint $p$. Consider the pullback lamination $\lam(q)$. Let
$G$ be the central symmetric gap or leaf  of $\lam(q)$ located between $\ell$
and $-\ell$. Then $G$ contains the origin and has leaves $\pm M$ closest to
criticality. Clearly, the short siblings $\pm d$ of leaves $\pm M$ form a
legal pair. Hence if $\ell$ shares an endpoint with $M$, then, by Lemma
\ref{pull-back-lam1}, we can set $c=d$. Assume now that leaves $\pm \ell$ are
disjoint from $\pm M$.

If the orbits of $p$ and $-p$ are disjoint, let $n=k$. Otherwise $k=2n$ for
some $n$, $\si_3^n(p)=-p$ and $\si_3^n(-p)=p$. We will assume in the rest of
the proof that $k=n$, the case when $k=2n$ is similar. Consider the
strip $S$ between $M$ and 
$M'$. If $s=\ol{x_0x_1}$ is the short pullback of $\ell$ included in
$\lam(q)$ by the construction, then $\si_3^k(s)=\ol{x_0p}$. Hence there is
another leaf $\ol{x_1x_2}$ such that $\si_3^k(\ol{x_1x_2})=\ol{x_0x_1}$. The
leaf $\ol{x_1x_2}$ is short as if $\ol{x_1x_2}$ is long/medium, then its
$k$-th image $s$ is short and non-disjoint from the interior of its short
strips, contradicting Lemma \ref{str-prop}. Repeating this, we get a
concatenation $A$ of 
pullbacks of $\ell$ under powers of $\si_3^k$; $A$ consists of short leaves
of $\lam(q)$, begins with $\ell\cup \ol{x_0x_1}\cup \ol{x_1x_2}$, converges
to a point $t\in \uc$ of period $k$, and points $x_0, x_1, \dots$ belong to
the short circular arc $I$ that bounds $S$ and does \emph{not} contain $p$.
Since $t$ and $p$ belong to distinct circle arcs on the boundary of $S$, then
$t\ne p$.

Clearly, an infinite periodic  gap $U$ of $\lam(q)$ contains $A$ in its
boundary, and there is a gap $U'$ with the same image as $U$ that shares an
edge $\ell$ with $U$. Consider the chord $\ol{pt}$; it is periodic of period
$k$, and there is another chord $\ol{x_0t'}$ with the same image as
$\ol{pt}$.
The chord $\ol{pt}$ is compatible with $\lam(q)$ because by
construction its images stay inside images of $U$ and never cross leaves of
$\lam(q)$.
Moreover, the iterated images of $\ol{pt}$ do not cross as for this
to happen some leaves from the concatenation $A$ must cross, and
this is not the case. We claim that then $\ol{pt}$ never enters the strip
between itself and $\ol{x_0t'}$. Indeed, if it does then, by Lemma
\ref{str-prop} it will have to cross $\ell$, a contradiction. Likewise,
images of $\ol{pt}$ never cross $-\ell$. By definition this implies that the
short sibling $\ol{qy}$ of $\ol{pt}$, together with $-\ol{qy}$, forms a legal
pair. Thus, $\ol{qy}=c$ is a comajor of a symmetric lamination as desired.

The leaf $\si_3(c)=\si_3(\ol{pt})$ is an $k$-periodic leaf of $\lam(q)$.
By
Proposition \ref{str-prop}, the leaf $\si^k_3(c)=\ol{pt}$ is a major of $\lam(c)$. Let
$\bar x$ and $\bar y$ be the two short edges of $Q_c=Q$. Removing them and
their backward orbits from $\lam(c)$ yields the family of chords $\hlam$; we
claim that $\hlam$ is an invariant lamination, too. Indeed, by definition
$\lam(c)$ has two quadrilaterals $X$ and $Y$ attached to $Q$ at $\bar x$ and
$\bar y$, respectively. This implies that both $\bar x$ and $\bar y$ are
isolated in $\lam(c)$. So, $\hlam$ is obtained by removing a countable family
of isolated leaves from $\lam$; hence, $\hlam$ is closed. The other
properties of invariant laminations for $\hlam$ are immediate. E.g., we need
to verify that any non-critical chord of $\hlam$ can be included in a sibling
collection. The only problematic case is that of $c$ (or $\ol{pt}$, or
$\ol{x_0t'}$), however $c,$ $\ol{pt},$ and $\ol{x_0t'}$ themselves form a
sibling collection. Thus, $\hlam$ is an invariant lamination. Evidently,
$\hlam$ is symmetric.

Consider the gap $U$ of $\hlam$ with $U\supset Q$. Countably many pullbacks
of $Q$ are consecutively attached to one another and contained in $U$. Hence
$U$ is an infinite periodic gap that maps forward $2$-to-$1$, that is,
 $U$ is a Fatou gap of degree two. 
By definition, $\hlam$ is hyperbolic.
Moreover, by the construction $c$ remains a leaf of $\hlam$. Hence $\{c,
-c\}$ is the comajor pair of $\hlam$.
\end{proof}

We now consider preperiodic points of preperiod greater than 1 or periodic points
(by Lemma \ref{cmajor-end-points}, there are no non-degenerate
periodic comajors).
Recall that a \emph{dendrite} is a locally connected
continuum that contains no Jordan curves. A q-lamination with no infinite
gaps gives rise to a topological Julia set which is a dendrite; we call
such q-laminations \emph{dendritic} (see \cite{bopt14,bopt19}). We will also
need Theorem 2.19 from \cite{bostv1}. This theorem coincides with Lemma 2.31
of \cite{bopt20} except for two extra claims proven in \cite{bostv1}

\begin{theorem}[Lemma 2.31\cite{bopt20}, Theorem
2.19\cite{bostv1}]\label{t:gaps-1}{\,}\\
Let $G$ be an infinite $n$-periodic gap and $K=\Bd(G)$. Then $\si_d^n|_K$
is the composition of a covering map and a monotone map of $K$. If
$\si_d^n|_K$ is of degree one, then either statement {\rm (1)} or statement
{\rm (2)} below holds.

\begin{enumerate}
\item The gap $G$ has countably many vertices, finitely many of which are
    periodic and the rest are preperiodic. All non-periodic edges of $G$
    are $($pre$)$critical and isolated. There is a critical edge with a
    periodic endpoint among edges of gaps from the orbit of $G$.

\item The map $\si_d^n|_K$ is monotonically semi-conjugate to an irrational
    circle rotation so that each fiber is a finite concatenation of
    $($pre$)$critical edges of $G$. Thus, there are critical leaves (edges
    of some images of $G$) with non-preperiodic endpoints.
\end{enumerate}

In particular, if all critical sets of a lamination are non-degenerate finite
polygons then the lamination has no infinite gaps.
\end{theorem}

Consider now the preperiodic case of preperiod greater than $1$.

\begin{Lemma}\label{l:misiu}
If $x\in \uc$ is preperiodic of preperiod $n>1$ then there exists a symmetric
dendritic q-lamination $\hlam$ with finite critical preperiodic sets $\pm G$
of preperiod $n$ and a 
 gap/leaf $T\ne \pm G$ of $\hlam$ with $\si_3(T)=\si_3(G)$ and
$x\in T$. Moreover,
\begin{enumerate}
  \item if $T$ is degenerate, then there are no non-degenerate
comajors containing $x$,
  \item if $T$ is a non-degenerate leaf, then $T$ is a comajor containing $x$, 
  \item if $T$ is a gap, then the edges of $T$ with endpoint $x$
are comajors containing $x$.
\end{enumerate}
The lamination $\hlam$ coincides with the family
of limit leaves of iterated pullbacks of critical leaves $\pm M_x$ of
$\lam(x)$. All edges of $T$ are comajors that are limits of comajors disjoint from $T$.
\end{Lemma}

\begin{proof}
Set $\lam(x)=\lam$, $\ell=M_x$. We claim that $\lam$ has no infinite gaps.
Indeed, if $U$ is an infinite gap of $\lam$, then by Theorem \ref{t:kiwi} an
eventual image $V$ of $U$ is periodic. Moreover, no gap of the orbit of $V$
is critical as $\lam$ has two critical leaves $\pm \ell$ and hence no gap of
$\lam$ can map onto its image  $k$-to-$1$ with $k>1$. Thus, $V$ is periodic
of degree $1$. By Theorem \ref{t:gaps-1} we may assume that $V$ has a
critical edge with a periodic endpoint or with both non-preperiodic
endpoints. Since neither $\ell$ nor $-\ell$ is like that, then all gaps of
$\lam$ are finite.

By Theorem \ref{t:nowander} the equivalence relation $\sim_{\lam}$ is
laminational. Let $\hlam$ be the q-lamination generated by $\sim_{\lam}$. All
gaps of $\hlam$ are finite (if $\widehat W$ is an infinite gap of $\hlam$
then by the construction no leaf of $\lam$ can be inside $\widehat W$, and so
there is an infinite gap $W$ of $\lam$ containing $\widehat W$, a
contradiction).
Hence the topological Julia set
$J_{\sim_{\lam}}$ is a \emph{dendrite}, and there are no isolated leaves in
$\hlam$. Clearly, $\hlam$ is symmetric, with critical sets $G\supset \ell,$
$-G\supset -\ell$, and there is a $\hlam$-set $T$ with $\si_3(T)=\si_3(G)$.

In order to prove claims (1) --- (3) of the lemma, assume first that $T=\{x\}$ is a
singleton. Then $\hlam$ has critical leaves $\pm G=\pm \ell$. Suppose that
there is a sequence of $\hlam$-gaps $H_i$ that converges to $\ell$. By
Theorem \ref{nwt-thm} all of them are (pre)periodic. We may assume that
$H_1=H$ has an edge $c$ that separates the interior of $H$ from $\ell$, with
endpoints
close to the 
endpoints of $\ell$.
We may follow the orbit of $H$ and $c$ and choose the
closest to criticality iterated image $d$ of $c$ (it is always possible since the
orbit of $c$ is finite and $c$ never maps to $\pm \ell$).
By Proposition \ref{str-prop}, the leaf $d$ never enters its short strips. Hence the short sibling
$d''$ of $d$, together with $-d''$, forms a legal pair. Evidently, $d''$
separates a short circle arc containing $x$ from the rest of the circle.
Since by Theorem \ref{t:oldmain} comajors form a q-lamination, non-degenerate
comajors cannot contain $x$ as claimed.

If there are no gaps located close to $\ell$ then, since
$\si_3$-periodic points are dense in $\uc$, we can choose a sequence of
periodic leaves converging to $\ell$, and repeat for them the above argument.
So, the case when $T=\{x\}$ is a singleton is considered. If $T$ is a
leaf/gap, then it is easy to check that any leaf on the boundary of $T$ is
legal as desired.

Let us prove the next to the last claim of the lemma. Take a leaf of $\lam(c)$ which is
the limit of a sequence of pullbacks of $\pm M_x$. Each such pullback is
contained in a pullback of a critical set of $\hlam$. Hence their limit is
the limit of a sequence of leaves of $\hlam$, that is itself a leaf of
$\hlam$. On the other hand, by definition $\hlam\subset \lam(x)$. Hence if
there is a leaf $\ell\in \hlam$ which is not the limit of a sequence of
pullbacks of $\pm M_x$, then $\ell$ is a pullback of $\pm M_x$. We may assume
that, say, $M_x$ is a leaf of $\hlam$ but is not the limit of pullbacks of
$\pm M_x$. Then there must exist two gaps of $\hlam$ sharing $M_x$ as an edge
which is impossible for the dendritic lamination $\hlam$ in which these two
gaps will have to be merged into one.

The last claim of the lemma follows from the construction and the fact that all
leaves of $T$ are limits of comajors disjoint from them proven in Lemma 6.8 in
\cite{bostv1} and stated in this paper as Lemma \ref{pre-period-bigger-than-1}.
\end{proof}

\begin{definition}\label{d:misiu}
A preperiodic comajor $c$ of preperiod greater than $1$ or a periodic comajor
 (necessarily degenerate)
is called a \emph{Misiurewicz comajor}, and any symmetric lamination with a
Misiurewicz comajor pair is said to be a \emph{Misiurewicz} symmetric
lamination.
\end{definition}

We will need the following lemmas.

\begin{Lemma}[Lemma 6.7 \cite{bostv1}]\label{appr-over}
Let $c\in C_sCL$ be a non-de\-ge\-ne\-rate comajor such that $\si_3(c)$ is
not periodic. If there exists a sequence of leaves $c_i\in \lam(c)$ with
$c\prec c_i$ and $c_i\rightarrow c,$ then $c$ is the limit of co-periodic
comajors $\hc_j \in \lam$ with $c\prec \hc_j$ for all $j$.
\end{Lemma}

\begin{Lemma}[Lemma 6.2 \cite{bostv1}]\label{appr-under}
Let $c\in C_sCL$ be a non-de\-ge\-ne\-rate comajor. If $\ell\in \lam(c),$
$\ell\prec c$ and $\|\ell\|>\frac{\|c\|}{3}$, then $\ell\in C_sCL$. In
particular, if $c_i\in \lam(c),$ $c_i\prec c$ and $c_i \rightarrow c$, then
$c_n\in C_sCL$ for sufficiently large $n$.
\end{Lemma}

We are ready to prove the density of hyperbolicity (Fatou conjecture) for
symmetric laminations.

\begin{theorem}\label{dense-cmajors}
Co-periodic comajors are dense in $C_sCL$.
\end{theorem}

\begin{proof}
Consider a non-degenerate comajor $c\in C_sCL$ that is not co-periodic.
We have two cases here.

(a) \textit{There is a sequence of leaves $c_i \in \lam(c)$ with $c\prec c_i$
and $c_i\rightarrow c$}.
Then, by Lemma \ref{appr-over}, the comajor $c$ is the limit of
co-periodic comajors $\hc_i$ such that $c\prec \hc_i$.

(b) \emph{A sequence of leaves $c_i\in \lam(c)$ converging to $c$ with
$c\prec c_i$ does not exist.} Then $c$ is an edge of a gap $G$ of $\lam(c)$
with all vertices of $G$ outside of $H(c)$. The lamination $\lam(c)$ has
critical quadrilaterals $\pm Q_c=\pm Q$. If $\si_3(c)$ eventually maps to an
edge of $Q$, then this edge is periodic which shows that $c$ is co-periodic,
a contradiction with our assumption. Hence $\si_3(c)$ never
maps to an edge of $Q$, and, therefore, $G$ never maps to a leaf or point. By
Theorem \ref{nwt-thm}, this implies that $G$ and $c$ are preperiodic of
preperiod greater than $1$ (recall that $c$ is not periodic by Lemma
\ref{cmajor-end-points}).

We claim that all edges of $G$ are comajors.
Properties of laminations imply that there are two gaps,
$L$ and $R$, attached to $Q_c$ at the appropriate majors of $\lam(c)$ and
such that $\si_3(L)=\si_3(R)=\si_3(G)$. Now, choose among the edges of $G$
the edge $\ell$ with the greatest length. Then, clearly, $G\cap \uc\subset
\ol{H(\ell)}$. Set $M=M_\ell, M'=M'_\ell$.
Then $M$ (or $M'$) cannot enter
the strip $S$ between $M$ and $M'$ as otherwise, by Proposition \ref{str-prop},
their images would have to cross edges of $L, R,$ or $Q_c$. This implies that
in fact any edge $d$ of $G$ is a comajor because $\{d, -d\}$ is legal.

It follows now that this is exactly the situation described in Lemma
\ref{l:misiu} and that $\lam(c)$ gives rise to a laminational equivalence
relation $\sim_{\lam(c)}$ which, in turn, gives rise to a dendritic
q-lamination $\hlam$ such that $G$ is a gap of $\hlam$ (the last claim
follows, e.g., from the fact that, by Theorem \ref{t:oldmain}, comajors form
a q-lamination). Since there are no isolated leaves in $\hlam$, the comajor
$c$ is approximated by uncountably many leaves $\hell$ of $\hlam$ such that $\hell\prec c$. By
Lemma \ref{appr-under}, we may assume that all these leaves of $\hlam$ are
comajors. Now, choose a sequence of them that converge to $c$ and satisfy the
conditions of case (a) of this proof. By (a) these leaves are all limits of
co-periodic comajors, hence so is $c$ as desired.
\end{proof}



\section{L-algorithm}

In this section, we provide an algorithm for constructing all co-periodic comajor leaves.
By Theorem \ref{dense-cmajors}, they are dense
in $C_sCL$, hence this renders the entire $C_sCL$. The algorithm is similar to the
famous \emph{Lavaurs} algorithm for Thurstons Quadratic Minor Lamination
$\qml$ \cite{lav86, lav89} (see \cite{sou21,bbs21} for an extension of this
algorithm to the degree $d$ unicritical case). We call it the \emph{L-algorithm}.

\subsection{Preliminaries} 

\begin{Lemma}[Lemma 6.1 \cite{bostv1}]\label{pre-period-1}
A co-periodic comajor leaf is disjoint from all other comajors in $C_sCL$.
\end{Lemma}

The following is Definition 6.4 from \cite{bostv1}.

\begin{definition}\label{d:symmetr}
Let $\ell$ be a leaf of a symmetric lamination $\lam$ and $k>0$ be such that
$\sigma_3^k(\ell)\neq\ell$ (in particular, the leaf $\ell$ is not a
diameter). If the leaf $\sigma_3^k(\ell)$ is under $\ell$, then we say that
the leaf $\ell$ \emph{moves in} by   $\sigma_3^k$; if $\sigma_3^k(\ell)$ is
not under $\ell$, then we say that the leaf $\ell$ \emph{moves out} by
$\sigma_3^k$. If two leaves $\ell$ and $\hell$ with $\ell \prec \hell$ of the
same lamination both move in or both move out by the map $\sigma_3^k$, then
we say that the leaves
 \emph{move in the same direction}. If one  of the leaves
$\ell$, $\hell$ moves in and the other moves out, then we say that the
leaves \emph{move in the  opposite directions}. There are two ways of moving
in the opposite directions: if  $\ell$ moves out and $\hell$ moves in, we say
they \textit{move towards each other}; if $\ell$ moves in and $\hell$ moves
out, we say that they  \textit{move away from each other}.
\end{definition}

The strip $\mathcal{S}(\ell, \hell)$ between non-crossing chords $\ell$,
 $\hell$ was introduced in Definition \ref{d:strip}.

\begin{Lemma}[Lemma 6.5 \cite{bostv1}]\label{move-twrds}
Let $\hell\neq\ell$ be non-periodic leaves of a symmetric lamination $\lam$
with $\hell\succ\ell$. Given an integer $k>0$,  let $h:\uc\to \uc$ be either
the map $\si_3^k$ or the map $-\si_3^k$. Suppose that the leaves $\ell$ and
$\hell$ move towards each other by the map $h$ and neither the leaves $\ell$
and $\hell$, nor any leaf separating them, can eventually map into a leaf
(including degenerate) with both endpoints in one of the boundary arcs of the
strip $\mathcal{S}(\ell,\hell)$. Then there exists a $\si_3$-periodic leaf
$y\in\lam$ that separates $\ell$ and $\hell$.
\end{Lemma}

For the notion of \emph{two-sided limit leaves}, see Definition \ref{d:2sided}.

\begin{corollary}[Corollary 6.7 \cite{bostv1}]\label{not-preperiodic}
Every not eventually periodic comajor $c$ is a two sided limit leaf in the
Cubic Symmetric Comajor Lamination $C_sCL$.
\end{corollary}


\begin{Lemma}[Lemma 6.8 \cite{bostv1}]\label{pre-period-bigger-than-1}
A non-degenerate preperiodic comajor $c$ of preperiod at least $2$ is a two
sided limit leaf of $C_sCL$ or an edge of a finite gap $H$ of $C_sCL$ all of whose
edges are limits of comajors of $C_sCL$ disjoint from $H$.
\end{Lemma}

\begin{definition}
We say a gap $G$ \emph{weakly separates} two leaves $\ell_1$ and $\ell_2$ if $\ell_1\sm
G$  and $\ell_2\sm G$ are nonempty sets in two different components of $\ol\disk \setminus G$.
Similarly we say a leaf $\ell$ \emph{weakly separates} two leaves $\ell_1$ and $\ell_2$
if $\ell_1\sm \ell$ and $\ell_2\sm \ell$ are nonempty sets in two different components of
$\ol\disk \setminus \ell$.
\end{definition}

\begin{Lemma}\label{moving-away-lem1}
Let $\ell'\neq\ell$ be two leaves in a cubic symmetric lamination $\lam$ such
that $\ell\prec \ell'$.
Suppose that: 

\begin{enumerate}
    \item the leaves $\ell$ and $\ell'$ move away from each other under 
    $\sigma_3^k$, 
    \item no leaf weakly separating $\ell$ and $\ell'$ maps to a critical chord of
        $\uc$ under the map $\sigma_3^i$ for $i<k$.
\end{enumerate}

Then, there exists a periodic leaf $y=\ol{ab}$ with $\si^k_3(a)=a,$
$\si^k_3(b)=b$ that weakly separates $\ell$ and $\ell'$.
\end{Lemma}

\begin{proof}
A gap $G$ of $\lam$ with edges $\ell, \ell'$ does not exist as otherwise the
gap $\sigma_3^k(G)$ would strictly cover the gap $G$. Hence the family of
leaves $\Ca\subset \lam$ that consists of $\ell$, $\ell'$, and the leaves
that weakly separate $\ell$ and $\ell'$ has at least one leaf that weakly separates $\ell$
and $\ell'$. Clearly, $\Ca$ is closed.

Let $A$ be the set of leaves of $\Ca$ that move in under $\sigma_3^k$ such
that for every leaf $m\in A$, if a leaf $n$ weakly separates $\ell$ and $m$, then
$n$ also moves in under the map $\sigma_3^k$. So, all the leaves in $A$ move
in under $\si_3^k$.
Then the closure $\ol{A}$ of $A$ (with respect to the Hausdorff metric) is a family of leaves, too; let $y\in
\ol{A}$ be the leaf of $\ol{A}$ farthest from $\ell$ (i.e., every leaf in
 $A\sm \{\ell, y\}$ weakly separates $\ell$ from $y$). By continuity, either $y\in A$,
or $\si_3^k(y)=y$. We claim that $\si_3^k(y)=y$. Indeed,  suppose that $y$ moves in under
$\si_3^k$.
There are two cases. 
First,  it can be that $y$ is approximated by leaves with endpoints outside $H(y)$ (see Definition \ref{under-dfn}).
However, this contradicts the choice of $y$.
Second, $y$ can be an edge of a gap $G$ with vertices outside of $H(y)$
while all vertices of $\si_3^k(G)$ belong to $\ol{H(y)}$.
If now $\hell$ is
the edge of $G$ with $y\prec \hell$, then $\hell\in A$, a contradiction.

Thus, $y=\ol{ab}=\sigma_3^k(y)$. We claim that $\sigma_3^k$ fixes the
endpoints of the leaf $y$. Assume that $\sigma_3^k$ flips $y$, and consider
cases. If $y$ is a two sided limit leaf and $t\in A$ is close to $y$, then
the leaf $t$ would move out under $\sigma_3^k$, a contradiction. If $y$ is an
edge of a gap $G$, then $y$ is an edge of the gap $G'=\sigma_3^k(G)$, the
gaps $G$ and $G'$ are on both sides of the leaf $y$, and $\sigma_3^k$ maps
one gap to the other.
Hence there is an edge $t\prec y$ of $G$ or $G'$ that
belongs to $A$ but moves out under $\si_3^k$, a contradiction. Finally,
$\si_3^k(y)=y$ is non-degenerate.
\end{proof}

\begin{Lemma}[Lemma 6.3 \cite{bostv1}]\label{shrt-leaf-pull-back}
Let $\lam$ be a cubic symmetric lamination with comajor pair $\{c,-c\}$.
Suppose that a short leaf $\ell_s\in \lam$ with $c\prec \ell_s$ is such that the
leaf $\ell_m=\sigma_3(\ell_s)$ never maps under $\pm\ell_m$.
Then, there
exists a cubic symmetric lamination $\lam(\ell_s)$ with comajor pair
$\{\ell_s,-\ell_s\}$.
\end{Lemma}

\subsection{The description of the L-algorithm}
According to \cite{mil93,mil09}, cubic polynomials with \emph{Fatou domains
whose first return map is of degree 4} are said to be of type B
(Bi-transitive) and cubic polynomials with \emph{two cycles of Fatou domains}
are said to be of type D (Disjoint); in the latter case first return maps on
periodic Fatou domains are, evidently, of degree 2.
We classify co-periodic comajors of in the similar fashion below.
Recall that, by Theorem
\ref{non-degen-cmajors}, co-periodic comajors $c$ generate hyperbolic q-laminations $\hlam(c)$.

The nature of cubic symmetric laminations gives rise to two notions
describing two types of periodic points and related (pre)periodic objects. We
give a general definition that applies to all of them.

\begin{definition}[Type B and type D]\label{d:bd}
A $2n$-periodic point $x$ of $\si_3$ such that $\si_3^n(x)=-x$, is said to be
\emph{of type B}. All other periodic points of $\si_3$ are said to be
\emph{of type D}. A periodic leaf of a symmetric lamination is of type B if
its endpoints are of type B, and of type D otherwise.
A co-periodic leaf of a symmetric lamination is of type B if its image is a
periodic leaf of type B, and of type D otherwise.
\end{definition}


\begin{Lemma}[Corollary 3.7 \cite{bmov13}]\label{l:37}
Suppose that $\ell$ and $\hell$ are two leaves of a $\si_d$-invariant
lamination that share an endpoint and have non-degenerate distinct images.
Then the orientation of the triple of their endpoints is preserved under the
map $\si_d$.
\end{Lemma}

To justify Definition \ref{d:bd} we need the next lemma.

\begin{Lemma}\label{l:bbdd}
A periodic leaf of a symmetric lamination $\lam$ cannot have endpoints of
type B and type D.
\end{Lemma}

\begin{proof}
Suppose that $\ell=\ol{xy}$ is a periodic leaf of $\lam$ such that $x$ is of
type B while $y$ is of type D. Then $x$ is of period $2n$ and
$\si_3^n(x)=-x$. It follows that $y$ is also of period $2n$ but
$\si_3^n(y)\ne -y$. Since $\ol{(-x) (\si_3^n(y))}=\si_3^n(\ell)$ is a leaf of
$\lam$, then the leaf $\ol{x (-\si_3^n(y))}$ is a leaf of $\lam$, too. Thus,
the cone of leaves $\ell=\ol{xy}$ and $\ol{x (-\si_3^n(y))}$ is mapped by
$\si_3^n$ to the cone of leaves $\ol{(-x) (\si_3^n(y))}$ and $\ol{(-x) (-y)}$.
However it is easy to see that the orientation of the triple $(y, x,
-\si_3^n(y))$
is opposite to the orientation of the
triple $(\si_3^n(y), -x, -y)$.
This contradicts Lemma
\ref{l:37} and completes the proof.
\end{proof}

Evidently, the $\si_3$-image of an object of type B (D) is an object of the
same type; co-periodic comajors can be either of type B or of type D.
Also,
Definition \ref{d:bd} allows us to talk about majors, comajors, and minors of
types B or D. In the type B case a periodic major $M=\ol{ab}$ eventually maps
to $-M$ so that $a$ and $b$ of $M$ eventually map to the $-a$ and $-b$,
respectively.
In the type D case, the orbits of majors are disjoint.
Thus, if a co-periodic comajor $c$ is of type B, then
the lamination $\hlam(c)$ from Theorem \ref{non-degen-cmajors}
has a pair of symmetric Fatou gaps whose first return map is of
degree $4$; if $c$ is of type D then $\hlam(c)$ has a pair of symmetric Fatou
gaps whose first return map is of degree 2.

\begin{definition}\label{d:ofnumber}
A periodic point (leaf) of type B and period $2n$ is said to be \emph{of
block period $n$}. A periodic point (leaf) of type D and period $n$ is said to be
\emph{of block period $n$}.
A co-periodic leaf is said to be
\emph{of block period $n$} if its image is of block period $n$.
\end{definition}

In \cite{bostv1} we considered the map $\tau$ that rotates the unit disk by
$180$ degrees.
If $\lam$ is a cubic symmetric lamination, then $\tau$ acts on
leaves and gaps of $\lam$. We will also interchangeably use the notation $-\ell$
for $\tau(\ell)$ and $-G$ for $\tau(G)$ where $\ell$ is a leaf of $\lam$ and
$G$ is a gap of $\lam$. Define the map $g_j=\tau\circ\sigma_3^j:\lam
\rightarrow\lam$ for some $j$. Lemma \ref{moving-away-lem2} is 
similar to Lemma \ref{moving-away-lem1}. 
We state it
without proof.

\begin{Lemma}\label{moving-away-lem2}
Let $\ell'\neq\ell$ be two leaves in a cubic symmetric lamination $\lam$ such
that $\ell\prec \ell'$. Suppose that: 

\begin{enumerate}
    \item[(i)] the leaves $\ell$ and $\ell'$ move away from each other
        under 
        $g_k$, 
    \item[(ii)] no leaf weakly separating $\ell$ and $\ell'$ maps to a critical
        chord of $\uc$ under the map $g_i$ for $i<k$.
\end{enumerate}

Then, there exists a periodic leaf $y$ of period $1$ under the map $g_k$ that
weakly separates $\ell$ and $\ell'$.
\end{Lemma}

The next lemma deals with dynamics of comajors.

\begin{Lemma}\label{l:distinct}
Suppose that $c'\prec c$ are distinct co-periodic comajors that are
leaves of a lamination $\lam$.
Then there is no finite gap $H$ of $\lam$ such that both
$c'$ and $c$ are edges of $H$. In particular, comajors of type B or
\end{Lemma}

\begin{proof}
The leaves $m'=\sigma_3(c')\prec m = \sigma_3(c)$  are periodic. By way of
contradiction assume that both are edges of a periodic gap $\si_3(H)=G$ of
$\lam$. Then their endpoints stay in the same circular order along their
periodic orbits. By Lemma \ref{fn-gaps}, if $G$ is \emph{1-rotational}, then
the leaf $m$ will eventually map to the leaf $m'$, and if $G$ is
\emph{2-rotational}, then the leaf $m$ will eventually map to the leaf $-m'$,
in either case contradicting that $m$ is the shortest leaf in its orbit (see
Lemma~\ref{short-leaves}).
\end{proof}

Now, the main theorem needed for the L-algorithm is as follows.

\begin{theorem}\label{lavaurs-like-alg}
Suppose that co-periodic comajors $c$ and $c'$ have the
following properties:

\begin{enumerate}
    \item[(i)] $c'\prec c$,
    \item[(ii)] both $c$ and $c'$ are either of type B or type D, and
    \item[(iii)] $c$ and $c'$ have the same block period $n$.
\end{enumerate}

Then there exists a co-periodic comajor $d$
with $c'\prec d\prec c$ such that $d$ is of block period $j<n$.
\end{theorem}

\begin{proof}
Choose a preperiodic point $p$ of preperiod bigger than 1 and period bigger
than $n$ in the arc $H(c')$. By Lemma \ref{l:misiu}, there exists a cubic
symmetric dendritic q-lamination $\lam$ with a pair of finite critical
gaps/leaves $\{\Delta,-\Delta\}$ such that $\si_3(p)\in \si_3(\Delta)$ (i.e.,
the critical leaves $\pm \ell$ of $\lam(p)$ are contained in the critical
sets $\Delta$ and $-\Delta$), iterated preimages of $\pm \ell$ converge to
all sides of $\Delta$ and $-\Delta$, so that pullbacks of the critical sets
are dense in $\lam$, and $c$ and $c'$ are leaves of $\lam$. The leaves $m =
\sigma_3(c)$ and $m'=\sigma_3(c')$ are periodic and such that $m'\prec m$.
Since preimages of $\pm \Delta$ are dense in $\lam$, then it follows from
Lemma \ref{l:distinct} that for a minimal $k$, the sets $\Delta$ or $-\Delta$
separates $\si_3^k(m)$ and $\si_3^k(m')$.
Consider cases.

(i): \textit{comajors $c$ and $c'$ are of type D}. Then the periodic orbits of $m$ and
$-m$ (and also $m'$ and $-m'$) are disjoint and have  period $n$. We claim
that $k\ne n-1$. If $k=n-1$, then $\sigma_3^{k-1}(m)$ and
$\sigma_3^{k-1}(m')$ are long/medium siblings of $c$ and $c'$, respectively.
Hence they must be separated by $\Delta$. The circular order of the four
endpoints of $m$ and $m'$ is preserved in the leaves $\sigma_3^{n-1}(m)$ and
$\sigma_3^{n-1}(m')$, but when $\si_3$ is applied one more time, exactly one
of the leaves $\sigma_3^{n-1}(m)$ and $\sigma_3^{n-1}(m')$ flips because of
the critical gap between them. Hence the order among the endpoints of
$\sigma_3^n(m)=m$ and $\sigma_3^n(m')=m'$ cannot be the same as the order
among the endpoints of $m$ and $m'$, which is absurd. Thus, $0<k<n-1$.

    \begin{figure}[h]
    \centering
    \includegraphics[width=0.58\linewidth]{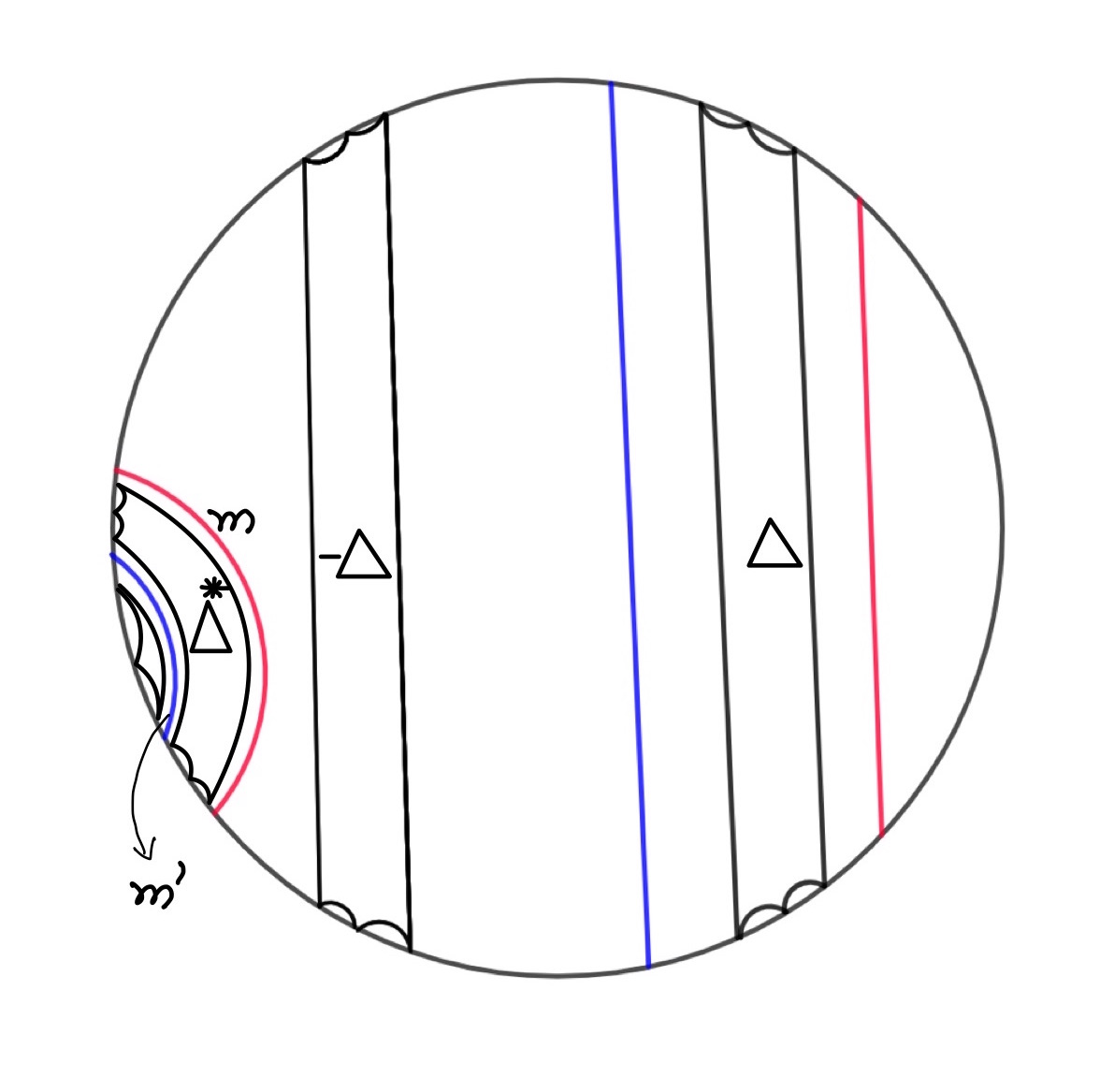}
    \caption{Cubic symmetric lamination $\lam$ with its gaps $\Delta, -\Delta$
    and $\Delta^{\ast}$ separating the leaves $m$ and $m'$ illustrating the proof of case (i).}
    \end{figure}{}

(a): \textit{it is $\Delta$ that separates the leaves $\sigma_3^k(m)$ and
$\sigma_3^k(m')$.} Since the leaves and gaps separating $m$ and $m'$ map
one-to-one under $\sigma_3^k$, there is a set $\Delta^{\ast}$ separating $m$
and $m'$ with $\sigma_3^k(\Delta^{\ast})=\Delta$ and
$\sigma_3^{k+1}(\Delta^{\ast})=\sigma_3(\Delta)\prec m'$. Let $\ell^{\ast}$
be the side of $\Delta^{\ast}$ that separates $m$ and $m'$ and is closest to
the leaf $m$. Then $\ell^{\ast}$ moves in under the map $\sigma_3^{k+1}$. On
the other hand, the leaf $\sigma_3^{k+1}(m)$ is neither under the leaf $m$
nor under the leaf $-m$ because the minor is the shortest leaf in its orbit.
Hence the leaves $m$ and $\ell^{\ast}$ move away from each other under the
map $\sigma_3^{k+1}$.


Let us verify condition (2) from Lemma
\ref{moving-away-lem1}.  Note that $\si_3^k(\ell^{\ast})=M$ is a major of $\lam$.
For $i\le k$, the map $\sigma_3^i$ takes the leaves separating $\ell^{\ast}$ and $m$
 in the strip $\mathcal{S}(\ell^{\ast}, m)$ one-to-one to the leaves
 separating $M$ and $\sigma_3^i(m)$ in the strip $\mathcal{S}(M,\sigma_3^i(m))$.
As there are no critical chords of $\uc$ in $\mathcal{S}(M,\sigma_3^i(m))$,
 no leaf separating $\ell^{\ast}$ and $m$ maps
to a critical chord of $\uc$ under the map $\sigma_3^i$ for $i\leq k$.
Hence,
by Lemma \ref{moving-away-lem1}, there is a periodic leaf $y\in \lam$ of
period $k+1<n$ separating $m$ and $\ell^{\ast}$.

Let $\Ca$ be the collection of the leaves separating $m$ and $m'$. 
Let $C_1$ be the collection of all $\si_3$-periodic leaves in $\Ca$ of period
smaller than $n$. Let $C_2$ be the collection of all fixed leaves under the
maps
$g_k=-\sigma_3^k,$ $0<k<n$ in $\Ca$; we associate the minimal such $k$ with
all leaves from $C_2$. Since $y\in C_1$, then $C_1\ne \emptyset$, but $C_2$
could be empty.


Let $y_1$ be a leaf of the least period $j_1\le k+1<n$ in $C_1$. Choose $y_1$
to be the closest to $m$ among leaves of $C_1$ of period $j_1$. Similarly,
choose a $-\sigma_3^{j_2}$-fixed leaf $y_2$ in $C_2$ such that $j_2$ is the
smallest possible; choose $y_2$ to be the closest to $m$ among
$-\sigma_3^{j_2}$-fixed leaves in $C_2$. If $j_1\leq j_2$, then we claim that
the leaf $d$ which is the short pullback of $y_1$ in $\lam$ is the desired
comajor of block period $j=j_1<n$ (recall that $y_1$ is located between the minors
$m$ and $m'$). By Lemma \ref{shrt-leaf-pull-back}, it suffices to prove that
the leaf $y_1$ neither maps under itself nor under the leaf $-y_1$ under the
map $\sigma_3^i$ where $i$ can be any block period smaller than $j_1$.


(1) If $y_1$ maps under itself under $\sigma_3^i$, for some $i<j_1$, then the leaves
$y_1$ and $m$ move away from each other under
    $\sigma_3^i$.
By Lemma \ref{moving-away-lem1}, there is a $\si_3$-periodic leaf $y_1'$ 
 of period $i<j_1$ separating $m$ and $y_1$; a contradiction with the minimality of
 $j_1$.


(2) If $y_1$ maps under $-y_1$ under
    $\sigma_3^i$ for some $i<j_1$, then the leaf $g_i(y_1)$ is under the
    leaf $y_1$. Now, the leaves $y_1$ and $m$ move away from each other
    under $g_i=-\sigma_3^i$. By Lemma \ref{moving-away-lem2}, there is
    a $-\sigma_3^i$-fixed leaf $y_1'$ that separates $m$ and $y_1$.
    Clearly $y_1'$ separates $m$ and $m'$, too. Then $i<j_1\le j_2$ is
    the block period associated with $y_1$, contradicting the choice of $j_2$.

Thus, the short pullback $d$ of  $y_1$ in $\lam$ is the desired comajor of
block period $j=j_1<n$. Similarly if $j_2<j_1$, then we obtain that the short pullback
$d$ of $y_2$ in $\lam$ is the desired comajor of block period $j=j_2<n$.

    \begin{figure}
    \centering
    \includegraphics[height=0.4\textheight]{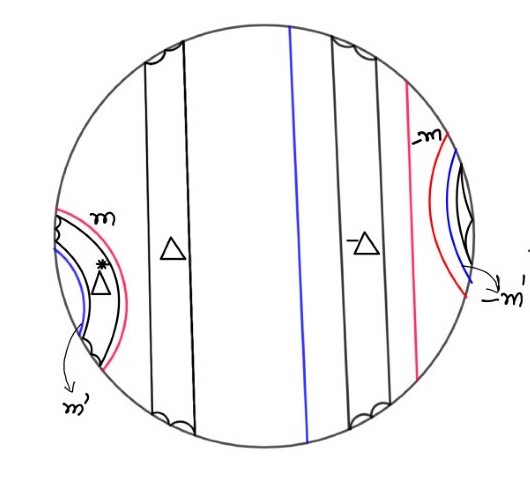}
    \caption{Cubic symmetric lamination $\lam$ with its gaps $\Delta, -\Delta$
    and $\Delta^{\ast}$ separating the leaves $m$ and $m'$ illustrating the proof of case (ii).}
\end{figure}{}

(b): \textit{it is $-\Delta$ that separates the leaves $\sigma_3^k(m)$ and
$\sigma_3^k(m')$, not $\Delta$.}
We use the arguments from case (a) and find a gap $\Delta^{\ast}$
 with $\sigma_3^{k}(\Delta^{\ast})=-\Delta$ separating $m$ and $m'$.
Then we have the gap $\sigma_3^{k+1}(\Delta^{\ast})$ going under the leaf
$-m'$. The only difference in the arguments is that we use Lemma
\ref{moving-away-lem2} first to find a leaf $y$ separating $m$ and $m'$ such
that $g_{k+1}(y)=y$. Thus, the collection $C_2$ is non-empty here whereas
collection $C_1$ could be empty. The rest of the argument follows exactly as
before and we end up with a comajor $d$ between $c$ and $c'$ of a block period
$j<n$.

(ii): \textit{comajors $c$ and $c'$ are of type B}. The leaves $m$ and $m'$ are now
periodic of period $p=2n$ and have symmetric orbits (the orbits of $m$ and
$-m$ are the same). Similarly, the orbits of the leaves $m'$ and $-m'$ are
the same as well.
In this case, the proof is very similar to that of case (i).

First, we show that there exists an integer $k$ with $0<k<n-1$ such
that $\Delta$ or $-\Delta$ separates the leaves
$\sigma_3^k(m)$ and $\sigma_3^k(m')$.
Indeed, let $k$ be the smallest integer between
$0$ and $p=2n$ such that 
the leaves $\sigma_3^k(m)$ and $\sigma_3^k(m')$ are separated by a critical gap/leaf.
As the orbits of
both the leaves $m$ and $m'$ are symmetric, the strips formed by the leaves
$\sigma_3^k(m)$ and $\sigma_3^i(m')$ where $0<i\leq n-1$ are symmetric to the
strips formed by the leaves $\sigma_3^r(m)$ and $\sigma_3^r(m')$ where $n\leq
r<2n$.
It follows that, for the first time, the separation by one of the critical
gaps/leaves $\Delta$ and $-\Delta$ happens during the first half of the cycle,
 i.e., $0<k\leq n-1$.

To see that $k$ cannot be equal to $n-1$, assume the contrary.
Since $\sigma_3^n(m)=-m$ and $\sigma_3^n(m')=-m'$,
 the leaves $\sigma_3^{n-1}(m)$ and
$\sigma_3^{n-1}(m')$ must be long/medium siblings of $-c$ and $-c'$,
respectively. Hence they are separated by $-\Delta$. The circular order of
the four endpoints of $m$ and $m'$ is preserved in the leaves
$\sigma_3^{n-1}(m)$ and $\sigma_3^{n-1}(m')$,
and exactly one of them flips under the next iteration
because of a critical gap between them.
Without loss of generality, assume that the leaf $\sigma_3^{n-1}(m)$
 flips its endpoints when it maps to the leaf
$-m=\sigma_3^n(m)$.
This contradicts the fact that 
 the endpoints of $m$ form two symmetric cycles rather than one cycle of period $4n$.
Thus, $0<k<n-1$.
We have two subcases here similar to case (i).

(a): \textit{it is $\Delta$ that separates the leaves $\sigma_3^k(m)$ and
$\sigma_3^k(m')$.} Then, following the similar arguments as in case(i)
part(a), we find a comajor $d$ of block period $j<k+1=n$ separating the leaves $c$ and $c'$

(b): \textit{it is $-\Delta$ that separates the leaves $\sigma_3^k(m)$ and
$\sigma_3^k(m')$, not $\Delta$.}
Then, using similar arguments to case(i) part(b), we find a comajor $d$ of block period $j<k+1=n$
 separating the leaves $c$ and $c'$.
\end{proof}

Theorem \ref{lavaurs-like-alg} allows us to describe an algorithm for finding
 co-periodic cubic comajors similar to the original Lavaurs algorithm \cite{lav86, lav89}
 for finding periodic quadratic minors.
We call this algorithm the \emph{L-algorithm}.

\begin{center} L-algorithm \end{center}

(1) Draw co-periodic comajors of block period $1$.
It is easy to verify that type D co-periodic comajors of period 1 are
$\overline{\frac{1}{6}\frac{1}{3}}$ and $\overline{\frac{2}{3}\frac{5}{6}}$.
Similarly, type B co-peiodic comajors of block period
$1$ are $\overline{\frac{5}{12}\frac{7}{12}}$ and
$\overline{\frac{11}{12}\frac{1}{12}}$.

We proceed by induction.
Suppose that all preperiodic comajors of block periods from $1$ to $k$
(inclusively) have been drawn. Denote the family of them by $\mathcal{F}_k$.
Consider a component $A$ of $\cdisk \sm \bigcup_{\ell\in \mathcal{F}_n} \ell$.
Then there are two cases.

(a) Suppose that there is a comajor $\ell_0$ such that all points of $A$ are
located under $\ell_0$. Then there may be several comajors $\ell_1,$ $\dots,$
$\ell_s$ located under $\ell_0$ with endpoints in $A\cap \uc$ (this collection
of comajors may be empty). Consider the set $B=\{b_1<\dots<b_t\}$ of
preperiodic points of type B of preperiod 1 and period $k+1$ that belong to
$A\cap \uc$. These points (if any) must be connected to create
several comajors. By Lemma \ref{pre-period-1} these comajors are pairwise
disjoint. By Theorem \ref{lavaurs-like-alg} no two comajors from that
collection can be located so that one of them is under the other one. Hence
$t=2r$ is even and the comajors in question are $\ol{b_1b_2},$ $\dots,$
$\ol{b_{2r-1}b_{2r}}$. We can also consider the set $D$ of preperiodic points
of type D of preperiod 1 and period $k+1$ that belong to $A\cap \uc$. These
points should be connected similar to how points from $B$ were connected,
i.e. \emph{consecutively}.

Do this for all components $A$ for which there is a comajor $\ell_0$ such that all points of $A$ are
located under $\ell$.

(b) There is exactly one component $C$ of $\cdisk \sm \bigcup_{\ell\in
\mathcal{F}_n} \ell$ for which there is no comajor $\ell_0$ with all points
of $A$ located under $\ell$. This is the ``central'' component left after the
closures of all components described in (a) are removed from $\disk$. Evidently, this
component contains the center of $\cdisk$ and is unique.

As before, let $B$ be the set of preperiodic points of type B of preperiod 1
and block period $k+1$ that belong to $C\cap \uc$. However, unlike before let us
divide $B$ into four subsets: $B^1=B\cap (\frac1{12}, \frac16)$, $B^2=B\cap
(\frac13, \frac5{12})$, $B_3=B\cap (\frac7{12}, \frac23),$ and $B_4=B\cap
(\frac{5}{6}, \frac{11}{12})$. Since comajors are short, a comajor cannot
connect two points from two distinct $B$-sets. Hence, as in case (a),
comajors connect points from $B$ consecutively within $B$-sets. If, e.g.,
$B_1=\{b_1<\dots<b_t\}$, then, as in (a), $t=2r$ is even, and the
corresponding comajors are $\ol{b_1b_2},$ $\dots,$ $\ol{b_{2r-1}b_{2r}}$.
Points of type D that belong to $\partial C$ should be treated similarly.

Thus, L-algorithm for cubic symmetric laminations is as follows. First, we
make step 1 and draw the comajors $\overline{\frac{11}{12}\frac{1}{12}},$
$\overline{\frac{1}{6}\frac{1}{3}},$ $\overline{\frac{5}{12}\frac{7}{12}},$
and $\overline{\frac{2}{3}\frac{5}{6}}.$ Then on each next step, say, $k+1,$
we first plot all type B points of preperiod 1 and period $k+1$
and connect them consecutively, starting from the smallest positive
angle. Then we plot all type D proints of preperiod 1 and
period $k+1$ and connect them consecutively, too, starting from the smallest
positive angle.

\noindent\textbf{Acknowledgments.} The results of this paper were reported by
the authors at the Lamination Seminar at UAB. It is a pleasure to express our
gratitude to the members of the seminar for their useful remarks and
suggestions.

\renewcommand\bibname{LIST OF REFERENCES}

\end{document}